\documentclass[a4,12pt]{amsart}
\usepackage{amssymb,amstext,amsmath,amscd,amsthm,amsfonts,enumerate,graphicx,latexsym}
\usepackage[usenames]{color}
\usepackage[all]{xy}
\newtheorem{thm}{Theorem}[section]
\newtheorem{cor}[thm]{Corollary}
\newtheorem{lem}[thm]{Lemma}
\newtheorem{prop}[thm]{Proposition}
\theoremstyle{definition}
\newtheorem{dfn}[thm]{Definition}
\newtheorem{rem}[thm]{Remark}

\newtheorem{ex}[thm]{Example}

\newtheorem*{claim*}{Claim}
\theoremstyle{remark}
\newtheorem*{ac}{Acknowlegments}

\numberwithin{equation}{thm}

\def\syz{\mathsf{\Omega}}

\def\depth{\operatorname{\mathsf{depth}}}

\def\Hom{\operatorname{\mathsf{Hom}}}
\def\depth{\operatorname{\mathsf{depth}}}

\def\gr{\operatorname{\mathsf{gr}}}
\def\index{\operatorname{\mathsf{index}}}
\def\reg{\operatorname{\mathsf{reg}}}
\begin{document}
\allowdisplaybreaks
\setlength{\baselineskip}{15pt}
\title[]{On delta invariants and indices of ideals}
\author{Toshinori Kobayashi}
\address{Graduate School of Mathematics, Nagoya University, Furocho, Chikusaku, Nagoya, Aichi 464-8602, Japan}
\email{m16021z@math.nagoya-u.ac.jp}
\thanks{2010 {\em Mathmatics subject  Classification.} 13A30, 13C14, 13H10}
\begin{abstract}
Let $R$ be a Cohen-Macaulay local ring with a canonical module. We consider Auslander's (higher) delta invariants of powers of certain ideals of $R$. Firstly, we shall provide some conditions for an ideal to be a parameter ideal in terms of delta invarints. As an application of this result, we give upper bounds for orders of Ulrich ideals of $R$ when $R$ has Gorenstein punctured spectrum. Secondly, we extend the definition of indices to the ideal case, and generalize the result of Avramov-Buchweitz-Iyengar-Miller on the relationship between the index and regularity.
\end{abstract}
\maketitle
\section{Introduction}
Let $(R,\mathfrak{m},k)$ be a Cohen-Macaulay local ring with a canonical module. The Auslander $\delta$-invariant $\delta_R(M)$ for a finitely generated $R$-module $M$ is defined to be the rank of maximal free summand of the a minimal Cohen-Macaulay approximation of $M$. For an integer $n\geq 0$, the $n$-th $\delta$-invariant is also defined by Auslander, Ding and Solberg \cite{AR} as $\delta_R^n(M)=\delta_R(\Omega^n_R M)$, where $\Omega^n_R M$ denotes the $n$ th syzygy module of $M$ in the minimal free resolution.

On these invariants, combining the Auslander's result (see\cite[Corollary 5.7]{AR}) and Yoshino's one\cite{Y}, we have the following theorem.

\begin{thm} [Auslander, Yoshino] \label{A0}
Let $d>0$ be the Krull dimension of $R$. Consider the following conditions.
\begin{itemize}
\item[(a)] $R$ is a regular local ring.
\item[(b)] There exists $n\geq 0$ such that $\delta^n(R/\mathfrak{m})>0$.
\item[(c)] There exists $n>0$ and $l>0$ such that $\delta^n(R/\mathfrak{m}^l)>0$.
\end{itemize}
Then, the implications {\rm (a)} $\Leftrightarrow$ {\rm(b)} $\Rightarrow$ {\rm(c)} hold. The implication {\rm(b)} $\Rightarrow$  {\rm(a)} holds if $\depth \gr_\mathfrak{m}(R)\geq d-1$.
\end{thm}

Here we denote by $\gr_I(R)$ the associated graded ring of $R$ with respect to an ideal $I$ of $R$. In this paper, we characterize parameter ideals in terms of (higher) $\delta$-invariants as follows.
\begin{thm} \label{A}
Let $(R,\mathfrak{m})$ be a Cohen-Macaulay local ring with a canonical module $\omega$, having a infinite residue field $k$ and Krull dimension $d>0$. Put $I$ be an $\mathfrak{m}$-primary ideal of $R$ such that $I/I^2$ is a free $R/I$-module. Consider the following conditions.
\begin{itemize}
\item[(a)] $\delta(R/I)>0$.
\item[(b)] $I$ is a parameter ideal of $R$.
\item[(c)] $\exists n\geq0$ such that $\delta^n(R/I)>0$.
\item[(d)] $\exists n>0$ and $l>0$ such that $\delta^n(R/I^l)>0$.
\end{itemize}
Then, the implications {\rm (a)} $\Rightarrow$ {\rm (b)} $\Leftrightarrow$ {\rm (c)} $\Rightarrow$ {\rm (d)} hold. The implication {\rm (d)} $\Rightarrow$ {\rm (c)} holds if $\depth \gr_I(R)\geq d-1$ and $I^i/I^{i+1}$ is a free $R/I$-module for any $i>0$. The implication {\rm (b)} $\Rightarrow$ {\rm (a)} also holds if $I \subset \mathsf{tr}(\omega)$.
\end{thm}

Here $\mathsf{tr}(\omega)$ is the trace ideal or $\omega$. that is, the image of the natural homomorphism $\omega\otimes_R \Hom_R(\omega,R) \to R$ mapping $x\otimes f$ to $f(x)$ for $x\in \omega$ and $f\in \Hom_R(\omega,R)$. This result recovers the Theorem \ref{A0} by letting $I=\mathfrak{m}$.

On the other hand, Ding studies the $\delta$-invariant of $R/\mathfrak{m}^l$ with $l\geq 1$ and defines the index $\index (R)$ of $R$ to be the smallest number $l$ such that $\delta(R/\mathfrak{m}^l)=1$.

Extending this, we define the index for an ideal.

\begin{dfn}
For an ideal $I$ of $R$, we define the {\it index} $\index(I)$ of $I$ to be the infimum of integers $l\geq 1$ such that $\delta_R(R/I^l)=1$.
\end{dfn}

The definition above recovers the ordinary index by taking the maximal ideal.

Using the argument of Ding \cite{D} concerning on indices of rings, Avramov, Buchweitz, Iyengar and Miller\cite[Lemma 1.5]{ABIM} showed the following equality.

\begin{thm}[Avramov-Buchweitz-Iyengar-Miller] \label{B0}
Assume that $R$ is a Gorenstein local ring and $\gr_\mathfrak{m}(R)$ is a Cohen-Macaulay graded ring. Then $\index(R)=\reg(\gr_\mathfrak{m}(R))+1$.
\end{thm}

The other main aim of this paper is to prove the following result.
\begin{thm} \label{B}
Let $R$ be a Cohen-Macaulay local ring having a canonical module and Krull dimension $d>0$, and $I$ be an ideal of $R$ such that $\gr_I(R)$ is a Cohen-Macaulay graded ring and $I^l/I^{l+1}$ is $R/I$-free for $1\leq l\leq \index I$. Then we have $\index I\geq \reg(\gr_I(R))+1$. The equality holds if $I\subset \mathsf{tr}(\omega)$.
\end{thm}

Note that this theorem recovera Theorem \ref{B0} by letting $I=\mathfrak{m}$.

There are some examples of ideals which satisfy the whole conditions in Theorem \ref{A} and \ref{B}. One of them is the maximal ideal $\mathfrak{m}$ in the case that $\gr_\mathfrak{m}(R)$ is Cohen-Macaulay (for example, $R$ is a hypersurface or a localization of a homogeneous graded ring.)

Other interesting examples are Ulrich ideals. These ideals are defined in \cite{GOTWY} and many examples of Ulrich ideals are given in \cite{GOTWY} and \cite{GOTWY2}. We shall show in Section 3 that Ulrich ideals satisfy the assumption of Theorem \ref{A} and \ref{B}.

We have an application of Theorem \ref{A} concerning Ulrich ideals as follows.

\begin{cor} \label{AAA}
Let $I$ be an Ulrich ideal of $R$ that is not a parameter ideal. Assume that $R$ is Gorenstein on the punctured spectrum. Then $I\not\subset \mathfrak{m}^{\index (R)}$. In particular, $\sup\{n \mid I\subset \mathfrak{m}^n\textrm{ for any Ulrich ideal $I$ that is not a parameter ideal}\}$ is finite.
\end{cor}

We prove this result in Section 3.

\section{Proofs}
Throughout this section, let $(R,\mathfrak{m}, k)$ be a Cohen-Macaulay local ring of dimension $d>0$ with a canonical module $\omega$, and assume that $k$ is infinite. We recall some basic properties of the Auslander $\delta$-invariant.

For a finitely generated $R$-module $M$, a short exact sequence
\begin{equation} \label{C}
0 \to Y \to X \xrightarrow[]{p} M \to 0 
\end{equation}
is called a Cohen-Macaulay approximation of $M$ if $X$ is a maximal Cohen-Macaulay $R$-module and $Y$ has finite injective dimension over $R$. We say that the sequence (\ref{C}) is minimal if each endomorphism $\phi$ of $X$ with $p\circ\phi=p$ is an automorphism of $X$. It is known (see \cite{AB}, \cite{HS}) that a minimal Cohen-Macaulay approximation of $M$ exists and is unique up to isomorphism.

If the sequence (\ref{C}) is a minimal Cohen-Macaulay approximation of $M$, then we define the (Auslander) $\delta$-invariant $\delta(M)$ of $M$ as the maximal rank of a free direct summand of $X$. We denote by $\delta^n(M)$ the $\delta$-invariant of $n$-th syzygy $\syz^n M$ of $M$ in the minimal free resolution for $n\geq 0$.

\begin{lem} \label{D}
Let $N$ be a maximal Cohen-Macaulay $R$-module. Then $\delta^1(N)=0$. In particular, $\delta^n(M)=0$ for $n\geq d+1$ and any finitely generated $R$-module $M$.
\end{lem}

\begin{proof}
Suppose that $\delta^1(N)>0$. Then $\syz N$ has a free direct summand. Let $\syz N=X\oplus R$. There is a short exact sequence $0 \to X\oplus R \xrightarrow[]{(\sigma, \tau)^T} R^{\oplus m} \xrightarrow[]{\pi} N \to 0$. According to \cite[Lemma 3.1]{T}, there exist exact sequences 
\begin{align}
\label{D1} 0 \to R \xrightarrow[]{\tau} R^{\oplus m} \to B \to 0,\\
\label{D2} 0\to R^{\oplus m} \to A\oplus B \to N \to 0
\end{align}
for some $R$-modules $A,B$. By the sequence (\ref{D2}), $B$ is a maximal Cohen-Macaulay $R$-module. In view of (\ref{D1}), $B$ is a free $R$-module provided that $B$ has finite projective dimension from (\ref{D1}). Then, the sequence (\ref{D1}) splits and $\tau$ has a left inverse map. This contradicts that the map $\pi$ is minimal.
\end{proof}

We now remark on $\delta$-invariants under reduction by a regular element. The following lemma is shown in \cite[Corollary 2.5]{K}.

\begin{lem} \label{E}
Let $M$ be a finitely generated $R$-module and $x\in\mathfrak{m}$ be a regular element on $M$ and $R$. If $0 \to Y \to X \to M \to 0$ is an minimal Cohen-Macaulay approximation of $M$, then
\[
0 \to Y/xY \to X/xX \to M/xM \to 0
\]
is a minimal Cohen-Macaulay approximation of $M/xM$ over $R/(x)$. In particular, it holds that $\delta_R(M)\leq \delta_{R/(x)}(M/xM)$.
\end{lem}

In the proofs of our theorems, the following lemma plays a key role. We remark that in the case $I=\mathfrak{m}$, similar statements are shown in \cite{D} and \cite{Y}.
\begin{lem} \label{F}
Let $l>0$ be an integer, $I$ be an $\mathfrak{m}$-primary ideal of $R$ and $x\in I\setminus I^2$ be an $R$-regular element. Assume that $I^i/I^{i+1}$ is a free $R/I$-module for any $1\leq i\leq l$ and the multiplication map $x:I^{i-1}/I^i\to I^i/I^{i+1}$ is injective for any $1\leq i\leq l$, where we set $I^0=R$. Then the following hold.
\begin{itemize}
\item[(1)] $xI^i=(x)\cap I^{i+1}$ for all $0\leq i\leq l$.
\item[(2)] $I^i/I^{i+1}\cong I^{i-1}/I^i\oplus I^i/(xI^{i-1}+I^{i+1})$ for all $1\leq i\leq l$.
\item[(3)] $I^i/xI^i\cong I^{i-1}/I^i\oplus I^i/xI^{i-1}$ for all $1\leq i\leq l$.
\item[(4)] $(I^i+(x))/xI^i\cong R/I^i\oplus I^i/xI^{i-1}$ for all $1\leq i\leq l$.
\item[(5)] $(I^i+(x))/x(I^i+(x))\cong R/(I^i+(x))\oplus I^i/xI^{i-1}$ for all $1\leq i\leq l$.
\end{itemize}
\end{lem}

\begin{proof}
(1): We prove this by induction on $i$. If $i=0$, there is nothing to prove. Let $i>0$. The injectivity of $x:I^{i-1}/I^i\to I^i/I^{i+1}$ shows that $xI^{i-1}\cap I^{i+1}=xI^i$. By the indcution hypothesis, $xI^{i-1}=(x)\cap I^i$. Thus it is seen that $xI^i=(x)\cap I^i(x)\cap I^i\cap I^{i+1}=(x)\cap I^{i+1}$.

(2): As $R/I$ is an Artinian ring, the injective map $x:I^{i-1}/I^i\to I^i/I^{i+1}$ of free $R/I$-modules is split injective. We can also see that the cokernel of this map is $I^i/(xI^{i-1}+I^{i+1})$. Therefore we have an isomorphism $I^i/I^{i+1}\cong I^{i-1}/I^i\oplus I^i/(xI^{i-1}+I^{i+1})$.

(3): We have the following natural commutative dialgram with exact rows:
\[
\xymatrix{
0 \ar@{->}[r]& I^{i-1}/I^i \ar@{->}[d]^=\ar@{->}[r]^x& I^i/xI^i \ar@{->}[d] \ar@{->}[r]& I^i/xI^{i-1} \ar@{->}[d] \ar@{->}[r]& 0\\
0 \ar@{->}[r]& I^{i-1}/I^i \ar@{->}[r]^x& I^i/I^{i+1} \ar@{->}[r]& I^i/(xI^{-1}+I^{i+1}) \ar@{->}[r]& 0
}
\]

We already saw in (2) that the second row is a split exact sequence, and thus the first row is also a split exact sequence. Therefore we have an isomorphism $I^i/xI^i\cong I^{i-1}/I^i\oplus I^i/xI^{i-1}$.

(4): The cokernel of the multiplication map $x:R/I^i\to (I^i+(x))/xI^i$ is $(I^i+(x))/(x)=I^i/((x)\cap I^i)$, which coincides with $I^i/xI^{i-1}$ by (1). Consider the following commutative diagram with exact rows:
\[
\xymatrix{
0 \ar@{->}[r]& I^{i-1}/I^i \ar@{->}[d]^{\iota_1}\ar@{->}[r]^x& I^i/xI^i \ar@{->}[d]^{\iota_2} \ar@{->}[r]& I^i/xI^{i-1} \ar@{->}[d]^= \ar@{->}[r]& 0\\
0 \ar@{->}[r]& R/I^i \ar@{->}[r]^-x&(I^i+(x))/xI^i\ar@{->}[r]& I^i/xI^{i-1} \ar@{->}[r]& 0
}
\]
Here $\iota_1,\iota_2$ are the natural inclutions. The first row is a split exact sequence as in (3). Therefore the second row is also a split exact sequence and we have an isomorphism $(I^i+(x))/xI^i\cong R/I^i\oplus I^i/xI^{i-1}$.

(5): The cokernel of the multiplication map $x:R/(I^i+(x))\to (I^i+(x))/x(I^i+(x))$ is $(I^i+(x))/(x)=I^i/xI^{i-1}$. We can get the following commutative diagram with exact rows:
\[
\xymatrix{
0 \ar@{->}[r]& R/I^i \ar@{->}[d]^{\pi_1}\ar@{->}[r]^-x&(I^i+(x))/xI^i\ar@{->}[d]^{\pi_2}\ar@{->}[r]& I^i/xI^{i-1} \ar@{->}[d]^=\ar@{->}[r]& 0\\
0 \ar@{->}[r]& R/(I^i+(x)) \ar@{->}[r]^-x&(I^i+(x))/x(I^i+(x))\ar@{->}[r]& I^i/xI^{i-1} \ar@{->}[r]& 0
}
\]
Here $\pi_1,\pi_2$ are the natural surjections. Then we can prove (5) in a manner similar to (4).
\end{proof}

In the case that the dimension $d$ is at most $1$, the $\delta$-invariants mostly vanish.

\begin{lem} \label{G}
Assume $d\leq 1$ and $I$ is an $\mathfrak{m}$-primary ideal of $R$. If $\delta(I)>0$, then $I$ is a parameter ideal of $R$.
\end{lem}

\begin{proof}
Since $d\leq 1$, the $\mathfrak{m}$-primary ideal $I$ is a maximal Cohen-Macaulay $R$-module. Therefore the condition $\delta(I)>0$ provides that $I$ has a free direct summand. We have $I=J+(x)$ and $J\cap (x)=0$ for some ideal $J$ and $R$-regular element $x\in I$. Let $y\in J$. Then $xy\in J\cap(x)=0$. Since $x$ is $R$-regular, the equality $xy=0$ implies $y=0$. This shows that $J=0$ and $I=(x)$.
\end{proof}

Now we can prove Theorem \ref{A}.

\begin{proof}[Proof of Theorem \ref{A}]
(b) $\Rightarrow$ (c): If $I$ is a parameter ideal, then $\syz^d (R/I)=R$ and hence $\delta^d(R/I)=1>0$.

(a), (c) $\Rightarrow$ (b): Assume that $\delta(R/I)>0$. Then the inequality $\delta(I)>0$ also holds because $I/I^2$ is a free $R/I$-module and thus there is a surjective homomorphism $I\to R/I$. Therefore we only need to prove the implication (c) $\Rightarrow$ (b)  in the case $n>0$. We show the implication by induction on the dimension $d$. 

If $d=1$, then $n=1$ by Lemma \ref{D}. Using Lemma \ref{G}, it follows that $I$ is a parameter ideal.

Now let $d>1$. Take $x\in I\setminus \mathfrak{m}I$ to be an $R$-regular element. Then the image of $x$ in the free $R/I$-module $I/I^2$ forms a part of a free basis over $R/I$. This provides that the map $x:R/I \to I/I^2$ is injective. We see from Lemma \ref{E} that 
\begin{align} \label{H}
\delta_{R/(x)}^{n-1}(I/xI)&=\delta_{R/(x)}(\syz^{n-1}_{R/(x)}(I/xI))\\
&=\delta_{R/(x)}(\syz^{n-1}_R(I)\otimes_R R/(x)) \notag\\
&\geq \delta_R(\syz_R^{n-1} I)=\delta_R^n(R/I)>0. \notag
\end{align}
Applying Lemma \ref{F} (3) to $i=1$, we have an isomorphism $I/xI\cong R/I\oplus I/(x)$ and hence we obtain an equality $\delta_{R/(x)}^{n-1}(I/xI)=\delta_{R/(x)}^{n-1}(R/I)+\delta_{R/(x)}^{n-1}(I/(x))$. It follows from (\ref{H}) that $\delta_{R/(x)}^{n-1}(R/I)>0$ or $\delta_{R/(x)}^{n-1}(I/(x))$. Note that the ideal $\overline{I}:=I/(x)$ of $\overline{R}:=R/(x)$ satisfies the same condition as (c), that is, the module $\overline{I}/\overline{I}^2$ is free over $\overline{R}/\overline{I}=R/I$, because $\overline{I}/\overline{I}^2=I/((x)+I^2)$ is a direct summand of $I/I^2$ by Lemma \ref{F} (2). By the induction hypothesis, the ideal $\overline{I}$ is a parameter ideal of $\overline{R}$. Then we see that $I$ is also a parameter ideal of $R$. 

(c) $\Rightarrow$ (d): This implication is trivial.

Next we prove by induction on $d$ the implication (d) $\Rightarrow$ (b) when $\depth \gr_I(R)\geq d-1$ and $I^i/I^{i+1}$ is a free $R/I$-module for any $i>0$. If $d=1$, then $\delta(I^l)>0$ by Lemma \ref{D}. By Lemma \ref{G}, it follows that $I^l$ is a parameter ideal.  Set $(y):=I^l$. Taking a minimal reduction $(t)$ of $I$, we have $I^{m+1}=tI^m$ for any $m \gg 0$. Setting $m=pl$, we obtain that $I\cong y^pI=I^{m+1}=tI^m=(ty^p)$. This shows that $I$ is a parameter ideal.

Assume $d>1$. Since $k$ is infinite, there is an element $x\in I\setminus I^2$ such that the initial form $x^*\in G$ is a non-zerodivisor of $G$. We see from Lemma \ref{E} that $\delta^{n-1}_{R/(x)}(I^l/xI^l)\geq \delta^n_R(R/I^l)>0$ in the same way as (\ref{H}). Applying Lemma \ref{F} (3), we get an isomorphism $I^l/xI^l\cong I^{l-1}/I^l\oplus I^l/xI^{l-1}$ and then we see that $\delta_{R/(x)}^{n-1}(I^l/xI^l)=\delta_{R/(x)}^{n-1}(I^{l-1}/I^l)+\delta_{R/(x)}^{n-1}(I^l/xI^{l-1})$. Since $I^{l-1}/I^l$ is a free $R/I$-module, we have $\delta_{R/(x)}^{n-1}(R/I)>0$ or  $\delta_{R/(x)}^{n-1}(I^l/xI^{l-1})>0$. In the case that $\delta_{R/(x)}^{n-1}(R/I)>0$, we already showed that $I$ is a parameter ideal. So we may assume that $\delta_{R/(x)}^{n-1}(I^l/xI^{l-1})>0$. The equality $xI^{l-1}=I^{l}\cap (x)$ in Lemma \ref{F} (1) shows that the image $\overline{I^l}$ of $I^l$ in $R/(x)$ coinsides with $I^l/xI^{l-1}$. Thus it holds that $\delta_{R/(x)}^{n-1}(\overline{I}^l)=\delta_{R/(x)}^{n-1}(\overline{I^l})>0$. We also note that $\overline{I}^i/\overline{I}^{i+1}$ is free over $\overline{R}/\overline{I}$ by Lemma \ref{F} (3). By the induction hypothesis, $\overline{I}$ is a parameter ideal of $R/(x)$. This implies that $I$ is also a parameter ideal of $R$.

Finaly, the implication (b) $\Rightarrow$ (a) follows from the proof of \cite[Theorem 11.42]{LW}.
\end{proof}

Next, to prove Theorem \ref{B}, we start by recalling the definition of regularity; see \cite[Definition 3]{O}.
\begin{dfn}
Let $A$ be a positively graded homogeneous ring and $M$ be a finitely generated graded $A$-module. Then the {\it (Castelnuovo-Mumford) regularity} of $M$ is defined by $\reg_A(M)=\sup\{i+j|H^i_{A_+}(M)_j\not=0\}$.
\end{dfn}

Here we state some properties of regularity.

\begin{rem}
Let $A$ and $M$ be the same as in the definition above.
\begin{itemize}
\item[(1)] Let $a\in A$ be a homogeneous $M$-regular element of degree $1$. Then we have $\reg_{A/(a)}(M)=\reg_A(M)$.
\item[(2)] If $A$ is an artinian ring, then $\reg(M)=\max\{p\mid M_p\not=0\}$. 
\end{itemize}
\end{rem}

Now let us state the proof of Theorem \ref{B}.

\begin{proof}[Proof of Theorem \ref{B}]
Since $k$ is infinite, there exists a regular sequence $x_1,\dots,x_d$ of $R$ in $I$ such that the sequence of initial forms $x_1^*,\dots,x_d^*$ makes a homogeneous system of parameters of $\gr_I(R)$. Then the eqularity $\gr_I(R)/(x_1^*,\dots,x_d^*)=\gr_{I'}(R')$ holds, where $R'=R/(x_1,\dots,x_d)$ and $I'=I/(x_1,\dots,x_d)$. It holds that $\reg (\gr_I(R))=\reg (\gr_{I'}(R'))=\max\{p|\gr_{I'}(R')_p\not=0\}=\max\{p|\gr_I(R)_p\not\subset (x_1^*,\dots,x_d^*)\}=\max\{p|I^p\not\subset(x_1,\dots,x_d)\}$. To show the inequality $\index(I)\geq \reg_I(R)+1$, it is enough to check that $I^p\subset (x_1,\dots,x_d)$ if $\delta(R/I^p)>0$. We prove this by induction on $d$. 

Let $\overline{R}$ be the quotient ring $R/(x_1)$ and $\overline{I}$ be the ideal $I/(x_1)$ of $\overline{R}$. Suppose $\delta_R(R/I^p)>0$. Then, since there is a surjection from $J:=I^p+(x)$ to $R/I^p$ by Lemma \ref{F} (4), $\delta_R(J)$ is greater than $0$ . Lemma \ref{E} yields that $\delta_{\overline{R}}(J/x_1J)\geq \delta_R(J)>0$. Using Lemma \ref{F} (5), we obtain an isomorphism $J/x_1J\cong R/J\oplus I^p/x_1I^{p-1}$, and hence $\delta_{\overline{R}}(J/x_1J)=\delta_{\overline{R}}(R/J)+\delta_{\overline{R}}(I^p/x_1I^{p-1})$. Therefore we see that $\delta_{\overline{R}}(R/J)>0$ or $\delta_{\overline{R}}(I^p/x_1I^{p-1})>0$. Now assume that $d=1$. If $\delta_{\overline{R}}(I^p/x_1I^{p-1})>0$, then $I^p/x_1I^{p-1}=\overline{R}$ since $I^p/x_1I^{p-1}=I^p/(x_1)\cap I^p$ is an ideal of the Artinian ring $\overline{R}$ and we apply Lemma \ref{G}. Therefore $I^p=R$ and this is a contradiction. So we get $\delta_{\overline{R}}(R/J)>0$. In this case, $R/J$ must have an $\overline{R}$-free summand. This shows that $J=(x_1)$ and $I^p\subset (x_1)$.

Next we assume that $d>1$. By Theorem \ref{A},  $\delta_{\overline{R}}(I^p/x_1I^{p-1})=0$. So we have $\delta_{\overline{R}}(R/J)>0$.  Then $R/J=R/(I^p+(x_1))=\overline{R}/\overline{I}^p$ hold. By the induction hypothesis, $\overline{I}^p\subset (x_2,\dots,x_d)/(x_1)$. Hence we get $I^p\subset (x_1,\dots,x_d)$.

It remains to show that $\index(I)=\reg(\gr_I(R))+1$ if $I\subset \mathsf{tr}(\omega)$. We only need to prove that $I^p\subset (x_1,\dots,x_d)$ implies $\delta(R/I^p)>0$. This immidiately follows from the inequalities $\delta(R/I^p)\geq \delta(R/(x_1,\dots,x_d))$ and $\delta(R/(x_1,\dots,x_d))>0$ by applying Theorem \ref{A} (b)$\Rightarrow$ (a) to the ideal $(x_1,\dots,x_d)$.
\end{proof}

\section{examples}
In this section, $R,\mathfrak{m},k),$ and $d$ are the same as in the previous section. Let $I$ be an $\mathfrak{m}$-primary ideal of $R$. To begin with, let us recall the definition of Ulrich ideals.

\begin{dfn} \label{AA}
We say that $I$ is an Ulrich ideal of $R$ if it satisfies the follwing.
\begin{itemize}
\item[(1)] $\gr_I(R)$ is a Cohen-Macaulay ring with $a(\gr_I(R))\leq 1-d$.
\item[(2)] $I/I^2$ is a free $R/I$-module.
\end{itemize}
\end{dfn}

Here we denote by $a(\gr_I(R))$ the $a$-invariant of $a(\gr_I(R))$. Since $k$ is infinite, the condition (1) of Definition \ref{AA} is equivalent to saying that $I^2=QI$ for some/any minimal reduction $Q$ of $I$.

Next, we prove that Ulrich ideals satisfies the assumption of Theorem \ref{A} and \ref{B}. 

\begin{prop}
Let $I$ be an Ulrich ideal of $R$. Then $I^l/I^{l+1}$ is a free $R/I$-module for any $l\geq 1$. 
\end{prop}

\begin{proof}
We use induction on $l$. By definition, $I/I^2$ is free over $R/I$. Let $l>1$, and take a minimal reduction $Q$ of $I$. Consider the canonical exact sequence $0 \to I^l/Q^l \to Q^{l-1}/Q^l \to Q^{l-1}/I^l \to 0$ of $R/Q$-modules. Then $Q^{l-1}/Q^l$ is a free $R/Q$-module and $Q^{l-1}/I^l=Q^{l-1}/IQ^{l-1}=R/I\otimes_{R/Q} Q^{l-1}/Q^l$ is a free $R/I$-module. Therefore $I^l/Q^l=\syz_{R/Q}((R/I)^{\oplus m})=\syz_{R/Q}(R/I)^{\oplus m}=(I/Q)^{\oplus m}$ for some $m$. Since $I/Q$ is free over $R/I$, $I^l/Q^l$ is also a free $R/I$-module. We now look at the canonical exact sequence $0 \to Q^l/I^{l+1} \to I^l/I^{l+1} \to I^l/Q^l \to 0$ of $R/I$-modules. Then as we already saw, $I^l/Q^l$ $Q^l/I^{l+1}$ are both free over $R/I$. Thus the sequence is split exact and $I^l/I^{l+1}$ is a free $R/I$-module.
\end{proof}

Now we give the proof of Corollary \ref{AAA}.

\begin{proof}[Proof of Corollary \ref{AAA}]
It follows from \cite[Theorem 1.1]{D2} that $\index(R)$ is finite number. Since $I$ is not a parameter ideal, we have $\delta(R/I)=0$ by Theorem \ref{A}. If $I\subset \mathfrak{m}^{\index(R)}$, then we have a surjective homomorphism $R/I \to R/\mathfrak{m}^{\index(R)}$ and thus $\delta(R/I)\geq  \delta(R/\mathfrak{m}^{\index(R)})>0$. This is a contradiction.
\end{proof}

To end this section, we give an example of an ideal showing that the condition $I^2/I^3$ is free over $R/I$ does not implies that $I^l/I^{l+1}$ is free over $R/I$ for any $l\geq 1$. 
\begin{ex}
Let $S=k[[x,y]]$ be the formal power series ring in two variables, $\mathfrak{n}$ be the maximal ideal of $S$, $L=(x^4)S$, $J=(x^2,y)S$, $R=S/L$ be the quotient ring of $S$ by $L$ and $I$ be the ideal $J/L$ of $R$. Then $I/I^2$ is free over $R/I$ but $I^2/I^3$ is not so.
\end{ex}

\begin{proof}
We note that $J$ is a parameter ideal of $S$ and therefore $J^l/J^{l+1}$ is free over $S/J$ for any $l\geq1$. Since $I^2=(J^2+L)/L=J^2/L$, we have $I/I^2=(J/L)/(J^2/L)=J/J^2$ which is free over $S/J=R/I$. On the other hand, $J^3\not\supset L$ so $I^2/I^3\not=J^2/J^3$. But $L\subset \mathfrak{n}J^2$, and hence the minimal number of generators $\mu_R(I^2)$ is equal to $\mu_S(J^2)$. If $I^2/I^3$ is a free $R/I$-module (hence free over $S/J$), then $I^2/I^3$ and $J^2/J^3$ have the same rank $\mu_R(I^2)$. However, there is a natural surjective homomorphism $\phi:J^2/J^3\to I^2/I^3$, so $\phi$ must be an isomorphism. This contradicts that $I^2/I^3\not=J^2/J^3$.
\end{proof}

\begin{ac}
The author is grateful to his supervisor Ryo Takahashi for giving him helpful advice throughout the paper.
\end{ac}


\end{document}